\documentclass[12pt]{amsart}

\usepackage{amsmath, amssymb,amsxtra}
\usepackage{amsthm}

\theoremstyle{plain} 
\newtheorem{theorem}{\sc Theorem}[section] 
\newtheorem{lemma}[theorem]{\sc Lemma}
\newtheorem{corollary}[theorem]{\sc Corollary}
\newtheorem{proposition}[theorem]{\sc Proposition}

\theoremstyle{definition} 
\newtheorem{definition}[theorem]{\sc Definition}
\newtheorem{remark}[theorem]{\sc Remark}
\newtheorem{example}[theorem]{\sc Example}

\makeatletter 

\title[]{Further study of modulation spaces as 
Banach algebras}
\author{
Hans G. Feichtinger, Masaharu Kobayashi, Enji Sato}

\address{
Hans G. Feichtinger \\
Faculty of Mathematics, 
University of Vienna, 
Oskar-Morgenstern-Platz 1, A-1090 Wien, Austria 
and
Acoustic Research Institute, 
OEAW, Vienna, Austria
}
\email{hans.feichtinger@univie.ac.at}

\address{
Masaharu Kobayashi \\
Department of Mathematics, 
Hokkaido University,
Kita 10, Nishi 8, Kita-Ku, Sapporo, 
Hokkaido, 060-0810, Japan}
\email{m-kobayashi@math.sci.hokudai.ac.jp}

\address{Enji Sato \\
Faculty of Science,
Yamagata University,
Kojirakawa 1-4-12, Yamagata-City,
Yamagata 990-8560,
Japan}
\email{esato@sci.kj.yamagata-u.ac.jp}

\makeatother
\keywords{Modulation spaces, Fourier-Wermer algebra,
Set of spectral synthesis, Segal algebra}
\subjclass[2010]{42B35, 43A45, 42B10}
\date{\today}

\begin{document}
\maketitle

\begin{center}
\textit{Dedicated to the 60th birthday of Professor Ferenc Weisz}
\end{center}

\begin{abstract}
This paper discusses 
spectral synthesis for those 
modulation spaces $M^{p,q}_s({\mathbf R}^n)$
which form Banach algebras under pointwise multiplication.
An important argument will be the ``ideal theory for  Segal algebras''   by H.~Reiter \cite{Reiter}. 
This paper is a continuation of our   paper \cite{Feichtinger-Kobayashi-Sato} where the case $q=1$ is treated.
As a by-product   we obtain a variant of  Wiener-L\'evy theorem for $M^{p,q}_s({\mathbf R}^n)$ and Fourier-Wermer algebras ${\mathcal F}L^q_s({\mathbf R}^n)$. 

\end{abstract}


\section{Introduction}
Spectral synthesis is one of the important topics
in classical harmonic analysis. It is concerned with the
question whether spectral synthesis holds
for a given subset $E$ of ${\mathbf R}^n$
and for a given Banach space $ (B,\|\cdot\|_B)$.
More precisely, we assume that  $B$ be a  Banach space of continuous, complex-valued functions, i.e. 
$B \subset C_b({\mathbf R}^n)$, 
with $C^\infty_c({\mathbf R}^n)$  dense in $B$, 
and convergence in $B$ implying pointwise convergence.
For every closed subset $E$ of ${\mathbf R}^n$,
we set $I(E) = \{ f \in B ~|~ f|_E=0 \}$
and define
$J(E)$ as the closure in $B$ of the set
$$
 \{ f \in C^\infty_c({\mathbf R}^n) ~|~ f=0 \
{\rm in~a~neighborhoof~of~} E \},
$$
where $f |_E=0$ means that $f(x)=0$ for all $x \in E$.
If $I(E) = J(E)$, then
$E$ is called a {\it set of spectral synthesis for} $B$.
There are many papers devoted to study the spectral synthesis
(e.g., \cite{Domar}, \cite{Katznelson}, \cite{Reiter}, \cite{Reiter-Stegeman 2000}, \cite{Rudin}).
Let  $A({\mathbf R}^n)$
be the Fourier algebra of all functions on ${\mathbf R}^n$ 
which are the Fourier transforms of 
functions in $L^1({\mathbf R}^n)$,
also denoted by $ {\mathcal F}L^1({\mathbf R}^n)$ elsewhere.
A famous and pioneering result by Schwartz \cite{Schwartz}
shows that  the unit sphere  $S^{n-1} = \{ x \in {\mathbf R}^n ~|~ |x| =1 \} $ 
is not a set of spectral synthesis for 
$A({\mathbf R}^n)$,
if $n \geq 3$.
Surprisingly, Herz \cite{Herz} showed that the unit  
circle  $S^1$ is a set of spectral synthesis for $A({\mathbf R}^2)$.
Motivated by these results, 
Reiter \cite{Reiter} showed similar result in the Fourier-Beurling algebras ${\mathcal F}L^1_s({\mathbf R}^n)$:
If $n \geq 3$, then $S^{n-1}$ is not a set of spectral synthesis for 
${\mathcal F}L^1_s ({\mathbf R}^n)$, and 
$S^1$ is a set of spectral synthesis for 
${\mathcal F}L^1_s({\mathbf R}^2)$ if $0 \leq s <  1/2$,
but not if $s \geq 1/2$.
Moreover, Kobayashi-Sato \cite{Kobayashi Sato 4} considered 
the spectral synthesis of $S^{n-1}$ for the  Fourier-Wermer algebra ${\mathcal F}L^q_s({\mathbf R}^n)$:
Let $q^\prime$ denote the conjugate exponents of $q$.
If $1<q<\infty$ and $s> n/q^\prime$, then $S^{n-1}$ $(n \geq 3)$ 
is not a set of spectral synthesis for 
${\mathcal F}L^q_s ({\mathbf R}^n)$, and 
$S^1$ is a set of spectral synthesis for 
${\mathcal F}L^q_s({\mathbf R}^2)$ if  $1<q<2$
and
$ 2/q^\prime  < s <  2/q^\prime +  1/2$,
but not if $1 <q< \infty$ and $s > 2/q^\prime + 1/2$.

The modulation spaces $M^{p,q}_s({\mathbf R}^n)$ 
are one of the function spaces
introduced by Feichtinger \cite{Feichtinger}.
The  definition of $M^{p,q}_s({\mathbf R}^n)$ will be given in Section \ref{def modulation spaces}. 
The main idea for these spaces is to consider 
the space and the frequency variable simultaneously. 
In some  sense, they behave like the Besov spaces $B^{p,q}_s({\mathbf R}^n)$. 
But they  appear to be better suited for the  description of
 problems in the area of time-frequency analysis
and are often a good substitute for the usual
spaces  $L^p({\mathbf R}^n)$ or $B^{p,q}_s({\mathbf R}^n)$
(see \cite{Feichtinger-Strohmer}, \cite{Grochenig 2001}).

The aim of this  paper is to understand certain 
sets of spectral synthesis for
$M^{p,q}_s({\mathbf R}^n)$.
Our first result is:
\begin{theorem}
\label{relation between set of spectral synthesis}
\label{spectral synthesis for Mpqs}
Let $1 \leq p \leq 2$. Suppose that
$q=1$ and $s \geq 0$, or $1<q \leq p^\prime$ and
$s> n / q^\prime$.
Then for  any compact subset  $K$ of ${\mathbf R}^n$,
$K$ is a set of spectral synthesis for $M^{p,q}_s({\mathbf R}^n)$
if and only if $K$ is a set of spectral synthesis for 
${\mathcal F}L^q_s({\mathbf R}^n)$.
\end{theorem}

\begin{remark}
The case $q=1$ and $s \geq 0$ is given in
\cite{Feichtinger-Kobayashi-Sato}.
\end{remark}

Theorem \ref{relation between set of spectral synthesis} allows to 
give some concrete examples of 
sets of spectral synthesis for $M^{p,q}_s({\mathbf R}^n)$
(cf. \cite{Kobayashi Sato 4} or 
Proposition \ref{one point} below).
\begin{example}
$(i)$ 
Let $1 \leq p \leq 2$, $1<q \leq 2$ and $n / q^\prime < s  < n / q^\prime +1 $.
Then single points of ${\mathbf R}^n$ are sets of spectral 
synthesis for $M^{p,q}_s ({\mathbf R}^n)$.
\,\, 
$(ii)$ Let $1 \leq p \leq 2$, $1<q< \infty$ and $s >
n/q^\prime  + 1$.
Then single points of ${\mathbf R}^n$ are not 
sets of spectral synthesis for  $M^{p,q}_s({\mathbf R}^n)$.
\end{example}

\begin{example}
$(i)$
Let $1 \leq p \leq 2$, $1<q \leq p^\prime$ and  
$ 2 / q^\prime <s< 2 / q^\prime  + 1 / 2$.
Then $S^1$ is a set of spectral
synthesis for $M^{p,q}_s({\mathbf R}^2)$.
\,\, 
$(ii)$
But for $1 \leq p \leq 2$, $1<q <2$ and  
$s >  2 / q^\prime  + 1 / 2$ $S^1$ is not a set of spectral
synthesis for $M^{p,q}_s({\mathbf R}^2)$.
\,\, 
$(iii)$
Let $1 \leq p \leq 2$, $1<q   \leq p^\prime  $ and  
$s > n / q^\prime$.
Then $S^{n-1}$ $(n \geq 3)$ is not a set of spectral
synthesis for $M^{p,q}_s({\mathbf R}^n)$.
\end{example}

As a by-product of Theorem \ref{relation between set of spectral synthesis}  we obtain a variant of  Wiener-L\'evy theorem for ${\mathcal F}L^q_s({\mathbf R}^n)$ and $M^{p,q}_s({\mathbf R}^n)$. 
\begin{theorem}
\label{wiener-levy for FLqs}
Given $1<q<\infty$, $s> n/q^\prime$ and a real-valued 
function $f  \in  {\mathcal F}L^q_s({\mathbf R}^n)$.
Suppose that $F$ is an analytic function on a  neighborhood 
of $f({\mathbf R}^n) \cup \{ 0 \}$ with $F(0)=0$.
Then there exists $g \in {\mathcal F}L^q_s({\mathbf R}^n)$
such that 
$g(x) =F(f (x))$.
\end{theorem}
\begin{theorem}
\label{wiener-levy for Mpqs}
Given $1<p<\infty$, 
$1<q<\infty$, $s> n/q^\prime$ and a real-valued 
function $f \in M^{p,q}_s({\mathbf R}^n)$.
Suppose that $F$ is an analytic function on a  neighborhood 
of $f({\mathbf R}^n) \cup \{ 0 \}$ with $F(0)=0$.
Then there exists $g \in M^{p,q}_s({\mathbf R}^n)$
such that $g(x) =F(f(x))$. 
\end{theorem}
The organization of this paper is as follows. 
After a preliminary section devoted to the definitions of
$M^{p,q}_s({\mathbf R}^n)$ and ${\mathcal F}L^q_s ({\mathbf R}^n)$
we prove Theorem \ref{relation between set of spectral synthesis}
in Section \ref{sectioin spectral synthesis}.
Theorems \ref{wiener-levy for FLqs} and
\ref{wiener-levy for Mpqs}
are treated in Section \ref{section Wiener Levy theorem}.

\section{Preliminaries}

The following notation will be used throughout this article. We use $C$ to denote various positive constants which may change from line to line.
We use the notation  $I \lesssim J$ if $I$ is bounded by a constant times $J$
and we denote $I \approx J$ if $I \lesssim J$ and $J \lesssim I$.
The closed ball with center $x_0 \in {\mathbf R}^n$
and radius $r>0$ is defined by
$B_r (x_0) := \{ x \in {\mathbf R}^n ~|~ |x-x_0| \leq r \}$.
Let $\langle x \rangle := (1+|x|^2)^{\frac{1}{2}}$
for $x \in {\mathbf R}^n$.
We define for $1 \leq p < \infty$
and $s \in {\mathbf R}$
$$
\| f \|_{L^p_s} :=
\Big(
\int_{{\mathbf R}^n}
\big(  \langle x \rangle^s  |f(x)|  \big)^p dx
\Big)^{\frac{1}{p}},
$$
and
$ \| f \|_{L^\infty_s} :=
{\rm ess.sup}_{x \in {\mathbf R}^n} \langle x \rangle^s   |f(x)|$.
We simply write $L^p ({\mathbf R}^n)$
instead of $L^p_0({\mathbf R}^n)$.
For $1 \leq p < \infty$, 
$p^\prime$ denotes
the conjugate exponent of $p$,
i.e., $1 / p + 1 / p^\prime =1$.
We write $C_c^\infty({\mathbf R}^n)$ for the space 
of complex-valued infinitely differentiable functions on ${\mathbf R}^n$ with compact support.
${\mathcal S}({\mathbf R}^n)$ denotes the Schwartz space of
complex-valued rapidly decreasing infinitely differentiable functions on ${\mathbf R}^n$
and ${\mathcal S}^\prime({\mathbf R}^n)$ denotes the space of tempered distributions.
The Fourier transform of $f \in L^1({\mathbf R}^n) $ is 
$ {\mathcal F} f (\xi) = \widehat{f} (\xi)  := \int_{{\mathbf R}^n}  f(x) e^{-ix \xi} dx.$
Similarly, the inverse Fourier transform of $h \in L^1({\mathbf R}^n)$ is ${\mathcal F}^{-1} h(x) :=
(2 \pi)^{-n}\widehat{h} (-x)$.
Recall that
$(f*g)^\wedge = \widehat{f}  \widehat{g}$ and
$(fg)^\wedge = (2 \pi)^{-n} ( \widehat{f} * \widehat{g}  )$,
where
$(f*g) (x) = \int_{{\mathbf R}^n} f(x-y) g(y) dy $.
For two Banach spaces $B_1$ and $B_2$, 
$B_1 \hookrightarrow B_2$ means that 
$B_1$ is continuously embedded into $B_2$.

\subsection{Fourier-Wermer algebra}
\label{subsection Fourier-Wermer algebra}
For $1\leq q<\infty$ and  $s \in {\mathbf R}$ write 
 ${\mathcal F}L_s^q({\mathbf R}^n) = {\mathcal F}L_s^q$  
for the space of all tempered distributions with 
the following  norm is finite: 
$$
\|  f  \|_{{\mathcal F}L_s^q} :=
\Big(  \int_{{\mathbf R}^n}
\Big(\langle \xi \rangle^s  | \widehat{f} (\xi)|   \Big)^q   d \xi \Big)^{\frac{1}{q}}.
$$ 
It is well known that
$({\mathcal F}L^q_s({\mathbf R}^n), \| \cdot \|_{{\mathcal F}L^q_s} )$
is a Banach space and ${\mathcal S}({\mathbf R}^n)$
is dense in ${\mathcal F}L^q_s({\mathbf R}^n)$.
If $q=1$ and $s \geq 0$, or $1<q< \infty$ and $s> n / q^\prime$,
then ${\mathcal F}L^q_s ({\mathbf R}^n)$ is a multiplication algebra, i.e., 
\begin{align}
\label{multiplication constant for FLqs}
\| fg  \|_{{\mathcal F}L^q_s} \leq c \| f \|_{{\mathcal F}L^q_s}  
\| g \|_{{\mathcal F}L^q_s},  \qquad f,g  \in {\mathcal F}L^q_s ({\mathbf R}^n)
\end{align}
for some $c  \geq 1$. We call  it
the {\it Fourier-Wermer algebra} owing to the fact that 
it is the  Fourier image of the convolution algebra
$L^q_s ({\mathbf R}^n)$
that was studied in the early  paper of Wermer \cite{Wermer}. 
Moreover,  if $q=1$ and $s \geq 0$, or $1<q< \infty$ and $s> n / q^\prime$, then
$f \in {\mathcal F}L^q_s ({\mathbf R}^n)$
implies  $\widehat{f} \in L^1({\mathbf R}^n)$
by the H\"older inequality.
Thus the Riemann-Lebesgue lemma shows
$f \in C({\mathbf R}^n)$ and vanishes at infinity, 
and  the inversion formula applies, giving  
$f(x) = {\mathcal F}^{-1}( {\widehat{f}} ) (x) $
for all $x \in {\mathbf R}^n$.

One can prove that  ${\mathcal F}L^q_s({\mathbf R}^n)$  possess
approximate units.
\begin{lemma}
\label{approximate unit for FLqs}
Given $1<q< \infty$ and  $s> \frac{n}{q^\prime}$, or
$q=1$ and $s \geq 0$.
Then for  $f \in {\mathcal F}L^q_s({\mathbf R}^n)$
and $\varepsilon >0$ there exists $\phi  \in C^\infty_c ({\mathbf R}^n)$
such that 
$
\| \phi f - f \|_{{\mathcal F}L^q_s} < \varepsilon
$. 
\end{lemma}
\begin{remark}
The case
$q=1$ and $s \geq 0$ is given in \cite[Proposition 1.6.14]{Reiter-Stegeman 2000}).
\end{remark}

\begin{proof} Let $\psi \in C^\infty_c({\mathbf R}^n)$ be such that $\psi(0)=1$.
For $0< \lambda<1$ we set
$ \psi_\lambda  (x) = \psi(\lambda x)$.
Since
$
(\widehat{ \psi_\lambda } * \widehat{f})(\xi)
= \int_{{\mathbf R}^n}
\widehat{\psi} (\eta) \widehat{f} (\xi - \lambda \eta) d \eta
$
and
$
1= \psi(0) = (2 \pi)^{-n} \int_{{\mathbf R}^n}  \widehat{\psi} (\eta) d \eta $,
we have
\begin{align*}
( \psi_\lambda f -f  )^\wedge (\xi)
=
(2 \pi)^{-n} 
\int_{{\mathbf R}^n} 
\widehat{\psi} (\eta)
\Big(
\widehat{f} (\xi - \lambda \eta) - \widehat{f}(\xi)
\Big) d \eta.
\end{align*}
Applying  the Minkowski inequality for integrals  we obtain that
\begin{align*}
\| \psi_\lambda f - f \|_{{\mathcal F}L^q_s}
&
\lesssim
\int_{{\mathbf R}^n} 
|\widehat{\psi} (\eta) | 
\Big\| \langle \cdot \rangle^{s} | \widehat{f} (\cdot - \lambda \eta) - \widehat{f}(\cdot)  | \Big\|_{L^q} d \eta.
\end{align*}
We note $0 < \lambda <1$.
By the submultiplicity of 
$\langle \cdot \rangle^s$ for $s \geq 0$ one has
\begin{align*}
\Big\| \langle \cdot \rangle^{s} \big| \widehat{f} (\cdot - \lambda \eta) - 
\widehat{f}(\cdot)  \big| \Big\|_{L^q} 
&
\lesssim 
\langle  \lambda \eta \rangle^s 
\| \langle \cdot - \lambda \eta  \rangle^s |\widehat{f} (\cdot - \lambda \eta)| \|_{L^q}
+
\| \langle \cdot \rangle^s |\widehat{f} (\cdot)| \|_{L^q} \\
&
\lesssim \langle \eta \rangle^s 
\| f \|_{{\mathcal F}L^q_s}
\end{align*}
Thus we can easily see $\lim_{\lambda \to 0} \| \psi_\lambda f -f \|_{{\mathcal F}L^q_s} =0$.
Hence, for any $\varepsilon>0$ there exists $\phi := \psi_{\lambda_0}$
for some $0< \lambda_0 <1$ such that
$\| \phi f - f \|_{{\mathcal F}L^q_s} < \varepsilon$.

\end{proof}

Furthermore, we have the following result. 
\begin{lemma}
\label{approximate unit for FLqs withf(x_0)=0}
Given $1<q< \infty$ and  $s> n / q^\prime$, or
$q=1$ and $s \geq 0$.
Suppose that 
$f \in {\mathcal F}L^q_s({\mathbf R}^n)$ and  $f(x_0) =0$
for some $x_0 \in {\mathbf R}^n$.
Then for $\varepsilon >0$
there exists $g \in C^\infty_c({\mathbf R}^n)$
such that
$\| f - g \|_{{\mathcal F}L^q_s} < \varepsilon$
and $g(x_0)=0$.
\end{lemma}

\begin{proof}
Let $\varepsilon >0$. 
By Lemma \ref{approximate unit for FLqs}
there exists $\phi  \in C^\infty_c ({\mathbf R}^n)$
such that 
$
\| \phi f - f \|_{{\mathcal F}L^q_s} < \varepsilon
$
and
$(\phi f)(x_0)  =  \phi(x_0) f(x_0) =0$.
Take  $\psi  \in C^\infty_c({\mathbf R}^n)$ with
$\psi(x)=1$ on ${\rm supp~}\phi$.
Since $\phi f \in {\mathcal F}L^q_s({\mathbf R}^n)$
and  ${\mathcal S}({\mathbf R}^n)$ is dense 
in ${\mathcal F}L^q_s({\mathbf R}^n)$, there exists $g_0 \in {\mathcal S}({\mathbf R}^n)$
such that 
$
\| \phi f - g_0 \|_{{\mathcal F}L^q_s} 
<  \varepsilon  /  3( \| \psi \|_{{\mathcal F}L^q_s}  +1 ) .
$ 
Then
$
g(x) :=  ( g_0(x) -g_0(x_0)  )\psi(x)
\in C^\infty_c({\mathbf R}^n)
$
and
$$
| g(x_0)  |
\leq
\| \phi f - g_0 \|_{L^\infty}
\lesssim\|  ( \phi f -g_0     )^\wedge   \|_{L^1}
\lesssim
\| \phi f - g_0 \|_{{\mathcal F}L^q_s},
$$
and thus $g(x_0) =0$.
By $\phi= \phi \psi$, $(\phi f) (x_0)=0$
and \eqref{multiplication constant for FLqs}
one has
\begin{align*}
&\| f - g \|_{{\mathcal F}L^q_s} 
= \| f  - \phi f  
+ \phi f \psi 
 - (\phi f) (x_0) \psi
- (g_0   - g_0(x_0)) \psi
\|_{{\mathcal F}L^q_s} \\
&
\leq \| f - \phi f \|_{{\mathcal F}L^q_s} 
+ \| \phi f -g_0 \|_{{\mathcal F}L^q_s} \| \psi \|_{{\mathcal F}L^q_s} 
+ | g(x_0) - (\phi f) (x_0)   | \| \psi \|_{{\mathcal F}L^q_s} \\
& 
\lesssim \| f - \phi f \|_{{\mathcal F}L^q_s} 
+ 2 \| \phi f -g_0 \|_{{\mathcal F}L^q_s}  \| \psi \|_{{\mathcal F}L^q_s}  
< \varepsilon.
\end{align*} 
\end{proof}

\begin{corollary} 
Given $1<q< \infty$ and  $s> n / q^\prime$, or
$q=1$ and $s \geq 0$, $f \in {\mathcal F}L^q_s({\mathbf R}^n)$ and $x_0 \in {\mathbf R}^n$.
Then for  $\varepsilon >0$
there exists $g \in C^\infty_c({\mathbf R}^n)$
such that
$g(x_0)=f(x_0)$ and 
$\| f - g \|_{{\mathcal F}L^q_s} < \varepsilon$. 
\end{corollary}

\begin{proof}
It suffices to prove the case $x_0=0$;
the other case is easy to see by considering $F(x) :=f(x+x_0)$
instead.
Let $\varepsilon >0$.
Then the proof of  Lemma \ref{approximate unit for FLqs} implies that
there exists $\phi \in C^\infty_c ({\mathbf R}^n)$ with
$\phi(0) =1$ and $\| \phi f - f \|_{{\mathcal F}L^q_s} < 
\varepsilon / 2$.
Set $f_0(x) := (f(x)-f(0)) \phi (x) \in {\mathcal F}L^q_s({\mathbf R}^n)$. 
Then $f_0(0)=0$.
By Lemma \ref{approximate unit for FLqs withf(x_0)=0}
there exists $g_0 \in C^\infty_c ({\mathbf R}^n)$
such that $g_0(0)=0$ and 
$\| f_0 - g_0 \|_{{\mathcal F}L^q_s} < \varepsilon / 2$.
Now set $g(x) :=g_0(x) + f(0) \phi (x)$.
Then $g(0)=f(0)$ and 
$\| f - g \|_{{\mathcal F}L^q_s} < \varepsilon$. 
\end{proof}

\subsection{Modulation spaces}
\label{def modulation spaces}
Let $1 \leq  p,q \leq \infty$, $s \in {\mathbf R}$
and  $\varphi \in {\mathcal S} ({\mathbf R}^n)$
be such that 
\begin{align}
\label{BUPU}
{\rm supp~}
\varphi \subset [-1,1]^n
\quad
and
\quad
\sum_{k \in {\mathbf Z}^n} 
\varphi (\xi -k) =1 
\quad
(\xi \in {\mathbf R}^n).
\end{align}
Then 
$M_s^{p,q}({\mathbf R}^n) = M_s^{p,q}$ 
consists of all $ f \in {\mathcal S}^\prime({\mathbf R}^n)$ such that the norm
$$
\| f \|_{M^{p,q}_s}
=
\Big(
\sum_{k \in {\mathbf Z}^n}
\langle k \rangle^{sq}
\Big(
\int_{{\mathbf R}^n} 
| \varphi(D-k) f (x) |^p dx
\Big)^{ \frac{q}{p} }
\Big)^{\frac{1}{q}}
$$
is finite, with obvious modifications if
$p$ or $q= \infty$.
Here 
$\varphi(D-k)f(x) = {\mathcal F}^{-1}(\varphi (\cdot -k) \widehat{f}  (\cdot)) (x) $.
It is well known  that  $M^{p,q}_s({\mathbf R}^n)$
is a multiplication algebra
if  $s > n/q^\prime$, or $q=1$ and $s \geq 0$.
We also recall a few basic properties of the
function in $M^{p,q}_s({\mathbf R}^n)$. Next we state some
auxiliary lemmata. 
\begin{lemma}[\cite{Feichtinger-Kobayashi-Sato}, \cite{Guo Fan Dashan Wu Zhao}]
\label{product estimate for Mpqs}
Given $1 \leq p< \infty$,
$q=1$ and $s \geq 0$, or $1<q< \infty$ and $s> n / q^\prime$.
Then 
$
\| fg \|_{M^{p,q}_s} \lesssim \| f \|_{M^{\infty,q}_s} \| g \|_{M^{p,q}_s}
$
for $f,g \in {\mathcal S}({\mathbf R}^n)$.
\end{lemma}
%
\begin{lemma}
\label{approximate unit for Mpqs}
Let $1 \leq p < \infty$. Suppose that $q=1$ and $s \geq 0$, or 
$1<q<\infty$ and $s> n / q^\prime$.
If  $f \in M^{p,q}_s({\mathbf R}^n)$,
then for any $\varepsilon >0$ there exists $\phi \in C^\infty_c({\mathbf R}^n)$
such that 
$\| \phi f   - f  \|_{M^{p,q}_s} < \varepsilon$.
\end{lemma}

\begin{proof}
The case $p=q=1$ and $s=0$ was considered in Bhimani-Ratnakumar
\cite[Proposition 3.14]{Bhimani Ratnakumar}
and their method still applies for  the remaining cases,
with only trifling changes (see \cite{Feichtinger-Kobayashi-Sato}
for more details).
\end{proof}
A useful inclusion relation 
(cf. \cite{Lu}, \cite{Okoudjou})
is stated next.  
\begin{lemma}
\label{inclusion Mpqs subset FLqs}
Let $1 \leq p \leq 2$. 
Suppose that $q=1$ and $s \geq 0$,
or $1 < q \leq p^\prime$ and 
$s> n / q^\prime$.
Then we have
$M^{p,q}_s({\mathbf R}^n) \hookrightarrow {\mathcal F}L^q_s({\mathbf R}^n)$.
\end{lemma}
\begin{proof}
We give the proof only for $1<p \leq 2$;
the case $p=1$ is similar.
Given $f \in M^{p,q}_s({\mathbf R}^n)$
and $\varphi \in C^\infty_c ({\mathbf R}^n)$ 
as in \eqref{BUPU} 
 there exists $N \in {\mathbf N}$  with 
$$
\chi_{k+ [-1,1]^n} (\xi)
=\sum_{|\ell|  \leq N } \varphi(\xi - (k+\ell))  \chi_{k+[-1,1]^n} (\xi)
\quad
(\xi \in {\mathbf R}^n)
$$
for all $k \in {\mathbf Z}^n$.
By the H\"older inequality with
$ 1 /  (p^\prime / q ) + 
1 / ( p^\prime /  (p^\prime -q) ) =1  $
and the Hausdorff-Young inequality we see
\begin{align*}
\| f \|_{{\mathcal F}L^q_s}^q
&\leq 
\sum_{k \in {\mathbf Z}^n} 
\int_{k+[-1,1]^n}
 \langle \xi \rangle^{sq} |\widehat{f}(\xi)|^q   d \xi \\
& \lesssim
 \sum_{k \in {\mathbf Z}^n} 
\langle k \rangle^{sq}
\int_{k+[-1,1]^n}
\Big|
\sum_{|\ell| \leq N} \varphi (\xi -(k+ \ell)) \widehat{f} (\xi)
\Big|^q d \xi \\
&\lesssim
\sum_{|\ell| \leq N} 
\sum_{k \in {\mathbf Z}^n}
\langle k \rangle^{sq}
\int_{k+ [-1,1]^n} 
|\varphi(\xi -(k+\ell)) \widehat{f} (\xi)  |^q d \xi \\
& \lesssim 
\sum_{|\ell| \leq N} 
\sum_{k \in {\mathbf Z}^n} 
\langle k \rangle^{sq}
\| \varphi(\cdot - (k+\ell)   )   \widehat{f} (\cdot) \|^q_{L^{p^\prime}} \\
& \lesssim \sum_{|\ell| \leq N} 
\sum_{k \in {\mathbf Z}^n} 
\langle k \rangle^{sq}
\| \varphi (D- (k+\ell) )f \|^q_{L^p} 
 \lesssim \| f \|_{M^{p,q}_s}^q.
\end{align*} 
\end{proof} 
\begin{lemma}
\label{inclusion relation (FLqs)_c}
Let $1 \leq p < \infty$. Suppose that 
 $q=1$ and $s \geq 0$, or  $1<q< \infty$ and $s> n / q^\prime$.
Set $( {\mathcal F}L^q_s )_c
:= \{ f \in {\mathcal F}L^q_s ({\mathbf R}^n) ~|~ {\rm supp~} f {\rm ~is~ compact} \}.
$
Then we have $({\mathcal F}L^q_s)_c \hookrightarrow M^{p,q}_s ({\mathbf R}^n)$.
\end{lemma}
\begin{proof}
Let $f \in ({\mathcal F}L^q_s)_c$ and 
$\phi \in {\mathcal S}({\mathbf R}^n)$ be 
such that $\phi(x)=1$ on ${\rm supp~}f$.
By Lemma \ref{product estimate for Mpqs}
we have
$$
\| f \|_{M^{p,q}_s} 
= \| f \phi \|_{M^{p,q}_s}
\lesssim \| f \|_{M^{\infty,q}_s} \| \phi \|_{M^{p,q}_s}.
$$
Moreover, let $\varphi \in {\mathcal S}({\mathbf R}^n)$
as in \eqref{BUPU}.
Then there exists $N \in {\mathbf N}$ such that
$$
\varphi(\xi -k) = \sum_{|\ell | \leq N } \varphi(\xi-k) \varphi(\xi - (k+ \ell)) 
\quad (\xi \in {\mathbf R}^n)
$$
for all $k \in {\mathbf Z}^n$.
Then  the Hausdorff-Young and  the H\"older inequality show that
\begin{align*}
\| f \|_{M^{\infty,q}_s}^q
&\lesssim
\sum_{k \in {\mathbf Z}^n}
\langle k \rangle^{sq} \| \varphi( \cdot -k) \widehat{f} \|_{L^1}^q \\ 
& \leq
\sum_{k \in {\mathbf Z}^n} 
\langle k \rangle^{sq}
\Big(
\sum_{|\ell| \leq N} 
\| \varphi( \cdot - (k+\ell) ) \|_{L^{q^\prime}}
\| \varphi(\cdot -k) 
\widehat{f} \|_{L^q}
\Big)^q
 \\
&
\lesssim
\sum_{k \in {\mathbf Z}^n}
\int_{k+ [-1,1]^n} 
\langle \xi \rangle^{sq}
| \varphi(\xi -k)  \widehat{f} (\xi) |^q d \xi 
\lesssim \| f \|_{{\mathcal F}L^q_s}^q,
\end{align*}
which implies the desired result. 
\end{proof}
\begin{remark}
\label{equivalent norm Mpqs}
There is another characterization of 
$M^{p,q}_s$ using the short-time Fourier
transform, i.e., 
for $\phi \in {\mathcal S} ({\mathbf R}^n) \setminus \{ 0 \} $,
we set: 
$$
V_\phi f (x, \xi)
:=
\langle f(t), \phi(t-x) e^{it \xi}  \rangle
=
\int_{{\mathbf R}^n}
f(t)  \overline{\phi(t-x)} e^{-i t \xi} dt, 
$$
$$
\| f \|_{M^{p,q}_s}
\approx
\Big(
\int_{{\mathbf R}^n}
\langle \xi  \rangle^{sq}
\Big(
\int_{{\mathbf R}^n} | V_\phi f ( x, \xi) |^p dx \Big)^{\frac{q}{p}}
d \xi \Big)^{\frac{1}{q}},
$$ 
i.e. this defines an equivalent norm, and in addition 
\begin{align}
\label{basic STFT}
V_\phi f (x, \xi)
&=
(2 \pi)^{-n}
 e^{- i x \xi}
V_{ \widehat{\phi}  } \widehat{f} (\xi, -x)
=
(2 \pi)^{-n}
e^{-i x \xi}
(f* M_\xi \phi^*) (x),
\end{align}
where $\phi^* (x) = \overline{\phi(-x)}$
(see \cite[Lemma 3.1.1]{Grochenig 2001}).
\end{remark}
\section{Spectral synthesis}
\label{sectioin spectral synthesis}
Throughout this section, $X$ stands for 
$M^{p,1}_s({\mathbf R}^n)$ ($1 \leq p < \infty$, $s \geq 0$),
$M^{p,q}_s({\mathbf R}^n)$
($1 \leq p \leq 2$, $1<q< \infty$, $s > n / q^\prime$),
${\mathcal F}L^1_s ({\mathbf R}^n)$  $(s \geq 0)$
or ${\mathcal F}L^q_s ({\mathbf R}^n)$  
($1<q<\infty$, $s> n / q^\prime$).
Moreover,  the closure of  $X_0 \subset X$
in $X$ will be denoted by $\overline{X_0}^{ \| \cdot \|_{X}}$.
\begin{definition}
Let $I$ be a linear subspace of $X$.
Then $I$ is called an {\it ideal} in $X$
if $fg \in I$
whenever $f \in X$
and $g \in I$. 
Moreover, if an ideal $I$ in $X$ is a closed subset of
$X$, then $I$ is called a
{\it closed ideal} in $X$.
For a subset $S$ of $X$,
the set $\bigcap_{\lambda \in \Lambda} I_\lambda$
is called  the ideal generated by $S$, where 
$\{ I_\lambda \}_{\lambda \in \Lambda}$
denoted the set of all ideals in $X$ containing $S$.
\end{definition}
\begin{definition}
	Let  $I$ be a closed ideal
	in   $X$.
	Then the {\it zero-set} of $I$
	is defined by $\displaystyle Z(I)
:=  \bigcap_{f \in I} f^{-1}(\{ 0 \}) $
with $f^{-1}(\{ 0 \}) := \{ x \in {\mathbf R}^2 ~|~ f(x) = 0\}  $.	
\end{definition}
We note that 
$Z(I)$ is a closed subset of $X$
if $I$ is a closed ideal in $X$.
In fact, if $f \in X$, then 
$f$ is continuous on ${\mathbf R}^n$ and thus
$f^{-1} (\{ 0 \})$ is a closed subset of ${\mathbf R}^n$.
We will write $f|_E =0$ if $f(x) =0$ for all $x \in E$.
\begin{lemma}
	\label{Basic lemma 1 closed ideal}
	Let $E$ be a closed subset of ${\mathbf R}^n$.
	Then 
	$
	I(E) := 	\{ f \in X ~|~ f |_E = 0 \} 	$
	is a closed ideal in $X$ with 	$E = Z(I(E)) $.
\end{lemma}
\begin{proof}
	We give the proof  only for the case $X= M^{p,q}_s({\mathbf R}^n)$;
	the same applies to other cases.
	It is clear that $I(E)$ is an ideal in $M^{p,q}_s ({\mathbf R}^n)$.
	To see $I(E)$ is closed,
	let $f \in M^{p,q}_s ({\mathbf R}^n)$, $\{ f_n \}^\infty_{n=1} \subset I(E)$
	and
	$ \| f_n - f \|_{M^{p,q}_s} \to 0$ $(n \to \infty)$.
	Since
	$$
	\| f_n -f \|_{L^\infty}
	\lesssim
	\| f_n -f \|_{M^{\infty,1}} 
	\lesssim \| f_n -f \|_{M^{p,1}}
	\lesssim 
	\| f_n - f \|_{M^{p,q}_s},
	$$
	we see that 
	$\{ f_n \}^\infty_{n=1}$  converges pointwise to $f$ on ${\mathbf R}^n$.
	Since $f_n |_E =0$, we have $f|_E=0$, and thus $f \in I(E)$.
	Hence  $I(E)$ is closed.
	Next we prove $E=Z(I(E))$. 
	Since $E \subset Z(I(E))$ is clear, we
	show $Z(I(E)) \subset E$.
	Suppose 
	$x_0  \not\in E$. 
	Since $E$ is closed
	and  $C^\infty_c  ({\mathbf R}^n) \subset M^{p,1}_s ({\mathbf R}^n) $, 
	there exists 
	$f \in M^{p,1}_s ({\mathbf R}^n)$ such that
	$f(x_0 ) =1$ and $f |_E =0$.
	Then $f \in I(E)$ and $f(x_0) \not= 0 $.
	Thus   $x_0 \not\in Z(I(E))$,
	which implies the desired result.  	
\end{proof}
\begin{definition}
Let $E \subset {\mathbf R}^n $ be  closed
and $I(E)$ be the set defined in Lemma 
\ref{Basic lemma 1 closed ideal}.
Define $J (E)$ by the closed ideal in $X$ generated by
$$
J_0 (E) :=\{ f \in X ~|~ f (x)= 0 ~{\rm in ~ a~neighborhood~ of}~E \}.
$$
Then $E$ is called a {\it set of spectral synthesis} for $X$
if  $I(E) = J(E)$. 
\end{definition}
\begin{remark}
$J(E)$ is the smallest closed ideal $I^\prime$
with $Z(I^\prime) =E$  (cf. \cite{Kobayashi Sato 4}).
\end{remark}  

\subsection{Spectral synthesis for ${\mathcal F}L^q_s$}

\begin{proposition}
\label{one point}
$(i)$
Let $1<q \leq 2$ and $n / q^\prime < s  < n / q^\prime +1 $.
Then single points of ${\mathbf R}^n$ are sets of spectral 
synthesis for ${\mathcal F}L^q_s ({\mathbf R}^n)$.\\
$(ii)$ Let $1<q< \infty$ and $s > n / q^\prime  + 1$.
Then single points of ${\mathbf R}^n$ are not 
sets of spectral synthesis for  ${\mathcal F}L^q_s({\mathbf R}^n)$.
\end{proposition}
\begin{remark}
It is well-known that 
single points of ${\mathbf R}^n$
are sets of spectral synthesis for
${\mathcal F}L^1_s({\mathbf R}^n)$,
if $0 \leq s <1$
(see \cite[Theorem 2.7.9]{Reiter-Stegeman 2000}).
\end{remark}
To prove Proposition \ref{one point},
we use a modification of  \cite[Lemma 6.3.6]{Reiter-Stegeman 2000}.
\begin{lemma}
\label{estimates for psi}
Let $1<q \leq 2$.
Suppose that
$\psi^{(1)}, \psi^{(2)} \in C^\infty_c ({\mathbf R}^n)$ 
be such that
${\rm supp~} \psi^{(j)} \subset B(0,R)$
for some $R>0$ $(j=1,2)$.
Set $\psi := \psi^{(1)} * \psi^{(2)}$.
Then for $n / q^\prime < s  < n / q^\prime +1 $
and $\vartheta \in {\mathbf R}^n$
\begin{align*}
\big\|
 \langle  \cdot  \rangle^s 
\big( \widehat{\psi}  ( \cdot - \vartheta) - \widehat{\psi} ( \cdot ) \big)
\big\|_{L^q}
&\leq C_\psi
\Big(
| \vartheta  |^{s- \frac{n}{q^\prime}}
\max_{|t|\leq R}|e^{i \vartheta  t}-1|^{1- (  s- \frac{n}{q^\prime}  )   }
+|   \vartheta  |^s \Big).
\end{align*}
\end{lemma}
\begin{proof}
We give the proof only for $1<q<2$: 
the case $q=2$ is similar.
We first  note that
since $\widehat{\psi} = \widehat{\psi^{(1)}} \cdot \widehat{\psi^{(2)}}$ we have
\begin{align*}
& 
\widehat{\psi} ( \xi-  \vartheta   ) - \widehat{\psi}(\xi)\\
&=
\Big( \widehat{\psi^{(1)}} (\xi-  \vartheta  )-
\widehat{\psi^{(1)}}(\xi) \Big) 
\widehat{\psi^{(2)}} ( \xi-  \vartheta   )
+\widehat{\psi^{(1)}}(\xi)
\Big( \widehat{\psi^{(2)}}
(\xi-   \vartheta  ) - \widehat{\psi^{(2)}}(\xi)\Big).
\end{align*} 
Then the H\"older inequality with 
$1 / (2/q) 
+ 1 /( 2/( 2-q)  )=1 $
and the Plancherel theorem show that
\begin{align*}
& 
\big\|
\langle \xi \rangle^s
\big( \widehat{\psi}  (\xi - \vartheta) 
- \widehat{\psi}(\xi) 
\big)
\big\|_{L^q({\mathbf R}^n_\xi)}
\\
&
\lesssim
\big\|
 \widehat{\psi^{(1)}}
( \xi -  \vartheta) -
\widehat{\psi^{(1)}}(\xi)  \big\|_{L^2({\mathbf R}^n_\xi) }
\big\|
\langle \xi \rangle^{s}
 \widehat{\psi^{(2)}}
( \xi  - \vartheta ) 
\big\|_{ L^{\frac{2q}{2-q} } ({\mathbf R}^n_\xi) }\\
& \quad +
\big\|
 \widehat{\psi^{(2)}}
( \xi -   \vartheta ) -
\widehat{\psi^{(2)}}(\xi)  \big\|_{L^2({\mathbf R}^n_\xi)}
\big\|
\langle \xi \rangle^{s}
\widehat{\psi^{(1)}}
( \xi  -  \vartheta  ) 
\big\|_{ L^{\frac{2q}{2-q} } ({\mathbf R}^n_\xi)  }\\
&=
\| {\mathcal F}_{x \to \xi}
[ (e^{ i \vartheta x}-1)\psi^{(1)} (x)] (\xi)   \|_{L^2({\mathbf R}^n_\xi)}
\langle 
 \vartheta 
\rangle^{s} 
\big\|  \langle \cdot \rangle^s
\widehat{\psi^{(2)}}  \big\|_{L^{ \frac{2q}{2-q}   }} \\
& \quad +
\| {\mathcal F}_{x \to \xi}
[ (e^{ i \vartheta x}-1)\psi^{(2)} (x)](\xi)   \|_{L^2({\mathbf R}^n_\xi)}
\langle 
 \vartheta 
\rangle^{s} 
\big\|  \langle \cdot \rangle^s
\widehat{\psi^{(1)}}   \big\|_{L^{ \frac{2q}{2-q}  }}.
\end{align*}
Note
$|e^{ i  \vartheta  x  } -1 | 
\leq \min \{ 2,    |   \vartheta  x  | \}$.
The Plancherel theorem shows 
$$
\| {\mathcal F}_{x \to \xi}
[ (e^{ i  \vartheta  x}  -1)\psi^{(j)} (x) ] (\xi)  \|_{L^2({\mathbf R}^n_\xi)}
=
(2 \pi)^{ \frac{n}{2} }
\|
 (e^{ i  \vartheta  x}-1)\psi^{(j)} (x)  \|_{L^2({\mathbf R}^n_x)}
$$
for $j=1,2$.
Since ${\rm supp~} \psi^{(j)}  \subset B_R(0)$ $(j=1,2)$,
we obtain
\begin{align*}
&
\| {\mathcal F}_{x \to \xi}
[ (e^{ i  \vartheta  x}  -1)\psi^{(j)} (x) ]  (\xi) \|_{L^2({\mathbf R}^n_\xi)} 
( 1 + |   \vartheta   |^s )
\\
&\lesssim
\Big( \max_{ |x| \leq R}|e^{i   \vartheta     x}   -1| \Big)^{1-(s- \frac{n}{q^\prime})}
\big\|  |e^{i    \vartheta    x}-1|^{s- \frac{n}{q^\prime}}
|\psi^{(j)}  (x)| \big\|_{L^2({\mathbf R}^n_x)}
+ \| \psi^{(j)}  \|_{L^2} 
|   \vartheta   |^s  \\
&\lesssim
|   \vartheta    |^{s- \frac{n}{q^\prime}}
\Big( \max_{ |x| \leq R}|e^{i   \vartheta    x}-1| \Big)^{1-(s- \frac{n}{q^\prime})}
\big\|
|x|^{s- \frac{n}{q^\prime}}|
\psi^{(j)} (x)| \big\|_{L^2({}\mathbf R)^n_x}
 + |   \vartheta    |^s,
\end{align*}
which yields the desired inequality.
\end{proof}
\begin{lemma}
\label{weighted Bhimani Ratnakumar lemma}

Let  $1<q \leq 2$,
$n / q^\prime < s < n / q^\prime + 1$, $f \in {\mathcal F}L^1_s ({\mathbf R}^n)$, 
$x_0 \in {\mathbf R}^n$
and $\varepsilon >0$.
Then there exists $\phi \in C_c^\infty ({\mathbf R}^n)$ such that
$(i)$ $\|  (f- f (x_0) ) \phi \|_{M^{1,q}_s} < \varepsilon$.
$(ii)$
$\phi(x) =1$ in some neighborhood of $x_0$.
\end{lemma}

\begin{proof}
Note that 
there exist $\psi^{(1)}, \psi^{(2)}  \in C^\infty_c({\mathbf R}^n)$
such that 
${\rm supp~}  \psi^{(j)} \subset B_2(0)  $ $(j=1,2)$ and 
$\psi := \psi^{(1)}  * \psi^{(2)} $ 
satisfies 
$\psi(x)=1$
on $B_1(0)$ and 
${\rm supp~}\psi \subset B_5(0) $
(see \cite{Feichtinger-Kobayashi-Sato}).
For $\lambda >5$, 
we set $\psi_\lambda(x) :=\psi(\lambda x)$ and 
$
 h^\lambda(x) := (f(x)-f(x_0)) \psi_\lambda(x-x_0).
$
If  $x \in {\rm supp~} h^\lambda$,
then $\psi (x-x_0) =1$. Thus 
$h^\lambda (x) = h^\lambda (x) \psi (x-x_0)$.
We consider the case $x_0=0$; 
the other case can be treated by 
considering $\psi(x-x_0)$ instead of $\psi(x)$.
Let $g_0 (t) := e^{- |t|^2/ 2}$ $(t \in {\mathbf R}^n)$.
Then
\eqref{basic STFT} 
shows
\begin{align*}
\| h^\lambda \|_{M^{1,q}_s}
&= \| \|  \langle \xi \rangle^s  V_{g_0} h^\lambda
 (x,\xi) \|_{L^1({\mathbf R}^n_x)}
\|_{L^q ({\mathbf R}^n_\xi)}  
\approx
\| \| \langle \xi \rangle^s
V_{ \widehat{g_0} } \widehat{ h^\lambda}  (\xi, -x) \|_{L^1({\mathbf R}^n_x)}
\|_{L^q({\mathbf R}^n_\xi)}.
\end{align*}
Since  $\widehat{g_0} =(2 \pi)^{ \frac{n}{2} }  g_0$, $g_0^* =g_0$ and
$h^\lambda (x) =h^\lambda (x) \psi (x)$,
we obtain by 
\eqref{basic STFT} 
\begin{align*}
V_{ \widehat{g_0}}   \widehat{ h^\lambda } (\xi, -x)
&= (2 \pi)^{ - \frac{n}{2} }
V_{g_0}  ( \widehat{h^\lambda} * \widehat{\psi}) (\xi, -x)
= (2 \pi)^{ - \frac{n}{2} }
e^{ i x \xi}(\widehat{h^\lambda} * \widehat{\psi} * M_{-x}g_0)(\xi).
\end{align*}
The Minkowski inequality for integrals 
and the Young inequality yield
\begin{align*}
\| h^\lambda \|_{M^{1,q}_s}
&
=
(2 \pi)^{-n}
\| \| \langle \xi \rangle^s
V_{ \widehat{g_0} } \widehat{  h^\lambda  }  (\xi, -x) \|_{L^1({\mathbf R}^n_x)}
\|_{L^q({\mathbf R}^n_\xi)} \\
&\lesssim
\| \| \langle \xi \rangle^s
(\widehat{h^\lambda} * \widehat{\psi} * M_{-x}g_0)(\xi)
 \|_{L^q({\mathbf R}^n_\xi)}
\|_{L^1({\mathbf R}^n_x)} \\
& \leq
\|
\| \langle \xi \rangle^s \widehat{ h^\lambda  } (\xi)
 \|_{L^q({\mathbf R}^n_\xi)}
\| \langle \xi \rangle^s (\widehat{\psi} *M_{-x} g_0) (\xi)
\|_{L^1 ({\mathbf R}^n_\xi)}
\|_{L^1({\mathbf R}^n_x)} \\
&=
\| \langle \cdot \rangle^s \widehat{ h^\lambda  }
 \|_{L^q}
\| \psi \|_{M^{1,1}_s}.
\end{align*}
Since
$
\widehat{ h^\lambda  } (\xi)
= (2 \pi)^{-n} (\widehat{ \psi_\lambda } * \widehat{f}) (\xi) 
- f(0) \widehat{ \psi_\lambda } (\xi)
$
and
$\widehat{ \psi_\lambda  } (\xi) 
=  \lambda^{-n}  \widehat{\psi}  ( \xi / \lambda )  $,
we have
\begin{align*}
\widehat{h^\lambda}(\xi)
&=
\frac{1}{(2 \pi)^n}
\int_{{\mathbf R}^n}
\widehat{f}( \eta) \widehat{\psi_\lambda} 
(\xi- \eta) d \eta -
\frac{1}{(2 \pi)^n}
\Big(
\int_{{\mathbf R}^n}
\widehat{f} (\eta) d \eta
\Big)  
\widehat{\psi_\lambda}(\xi) \\
&=
\frac{1}{ (2 \pi \lambda)^n}
\int_{{\mathbf R}^n}
\widehat{f} (\eta) 
\Big(
\widehat{\psi}
\Big( \frac{\xi - \eta}  {\lambda}
\Big) - \widehat{\psi} \Big( \frac{\xi}{\lambda} \Big)
\Big) d  \eta.
\end{align*}
Then the Minkowski inequality for integrals 
and
Lemma \ref{estimates for psi} 
choosing $ \vartheta =  \eta /  \lambda$
and $R=4$
yield that
\begin{align*}
\| \langle \cdot \rangle^s \widehat{ h^\lambda  } 
 \|_{L^q({\mathbf R}^n)} 
&\leq
\lambda^{-n}
\int_{{\mathbf R}^n}
\Big\|
 \langle  \cdot  \rangle^s 
\Big( \widehat{\psi}  \Big(  \frac{\cdot - \eta}
{\lambda}  \Big) - \widehat{\psi} \Big( \frac{\cdot}{\lambda} \Big) \Big)
\Big\|_{L^q}
 |\widehat{f}(\eta)|
 d \eta  \\
&
\lesssim
\lambda^{-n + s+ \frac{n}{q} }
\int_{{\mathbf R}^n}
\Big\|
 \langle  \cdot  \rangle^s 
\Big( \widehat{\psi}  \Big( \cdot -  \frac{ \eta}
{\lambda}  \Big) - \widehat{\psi} ( \cdot ) \Big)
\Big\|_{L^q}
 |\widehat{f}(\eta)|
 d \eta  
\\
&
\lesssim
\lambda^{ s- \frac{n}{q^\prime} }
\int_{{\mathbf R}^n}
\Big(
\Big| \frac{\eta}{\lambda}  \Big|^{s- \frac{n}{q^\prime}}
\Big( \max_{|t|\leq 4}|e^{i\frac{\eta}{\lambda}t}-1| \Big)^{1- (  s- \frac{n}{q^\prime}  )   }
+ \Big|   \frac{\eta}{\lambda}   \Big|^s \Big)
 |\widehat{f} (\eta)| d \eta \\
& \lesssim
\int_{{\mathbf R}^n}
\Big(
\Big(
\max_{|t|\leq 4}|e^{i\frac{\eta}{\lambda}t}-1|
\Big)^{1-(s- \frac{n}{q^\prime} )}
+ \lambda^{- \frac{n}{q^\prime} }
\Big)
\langle \eta \rangle^s
|\widehat{f} (\eta)|
d \eta.
\end{align*}
We observe that  
$\left(
\max_{|t|\leq 4}|e^{i\frac{\eta}{\lambda}t}-1|
\right)^{1-(s- \frac{n}{q^\prime})}  + \lambda^{ - \frac{n}{q^\prime} } \leq 3$ and
$$
\Big(
\max_{|t|\leq 4}|e^{i\frac{\eta}{\lambda}t}-1|
\Big)^{1-(s- \frac{n}{q^\prime})}  + \lambda^{- \frac{n}{q^\prime}} \to 0
\quad
(\lambda \to \infty).
$$
Since $f \in {\mathcal F}L^1_s ({\mathbf R}^n)$,
we see that
$\| \langle \cdot \rangle^s \widehat{ h^\lambda  } 
 \|_{L^q}
\to 0 $
$(\lambda \to \infty)$.
Therefore,
for any $\varepsilon >0$, there exists
$\lambda_0 >0$ such that
$
\| h^{\lambda_0} \|_{M^{1,q}_s} <   \varepsilon$.
Hence,
by putting  $\phi (x)=\psi_{\lambda_0} (x)$,
we have the desired result.
\end{proof}
\begin{remark}
\label{f-f(x0) psi estimate for FLqs}
Let $1<q \leq 2$, $n / q^\prime <s<  n / q^\prime +1 $,
$f \in {\mathcal F}L^1_s ({\mathbf R}^n)$
and $x_0 \in {\mathbf R}^n$. 
Since  $M^{1,q}_s({\mathbf R}^n)  \hookrightarrow {\mathcal F}L^q_s({\mathbf R}^n)$,
Lemma \ref{weighted Bhimani Ratnakumar lemma} implies
that
for any $\varepsilon >0$, there exists $\phi \in C^\infty_c ({\mathbf R}^n)$
such that 
$(i)$ $\| (f- f(x_0) ) \phi \|_{{\mathcal F}L^q_s} < \varepsilon$.
$(ii)$ $\phi(x) =1$ in some neighborhood of $x_0$.
\end{remark}

\subsubsection{The proof of Proposition \ref{one point}}
\noindent
$(i)$
It suffices to show
$ I(\{ x_0 \}) \subset  \overline{  J_0(\{ x_0 \}) }^{\| \cdot \|_{{\mathcal F}L^q_s}} = J(\{ x_0 \})$.
Let $f \in I(\{ x_0 \})$ and $\varepsilon >0$. 
By Lemma \ref{approximate unit for FLqs withf(x_0)=0}
there exists $g \in C^\infty_c ({\mathbf R}^n)$
such that $g(x_0)=0$ and
$\| f -g \|_{{\mathcal F}L^q_s} < 
\varepsilon / 2$.
Moreover, by Remark \ref{f-f(x0) psi estimate for FLqs}
there exists $\phi \in C^\infty_c ({\mathbf R}^n)$ such that
$\| (g - g(x_0)) \phi \|_{{\mathcal F}L^q_s} = \|  g \phi \|_{{\mathcal F}L^q_s} 
< \varepsilon / 2$
and $\phi(x) =1$ on a neighborhood of $x_0$.
Let $\tau= (1- \phi) g$.
Then 
$\tau \in {\mathcal F}L^q_s ({\mathbf R}^n)$
and 
$\tau(x) =0$ on a neighborhood of $x_0$.
Therefore $\tau \in J_0 (\{ x_0 \})$ and
$
\| f -\tau \|_{{\mathcal F}L^q_s}
 \leq
\| f -g \|_{{\mathcal F}L^q_s} + \| \phi g \|_{{\mathcal F}L^q_s}
< \varepsilon,
$
which implies
$
f \in \overline{  J_0(\{ x_0 \}) }^{\| \cdot \|_{{\mathcal F}L^q_s}}.
$

\noindent
$(ii)$ Let $x_0 = (x^{(0)}_1, \cdots, x_n^{(0)}) \in {\mathbf R}^n$.
Contrary to our claim, 
suppose that $\{ x_0 \}$ is a set of spectral synthesis 
for ${\mathcal F}L^q_s({\mathbf R}^n)$.
Let $\phi \in C^\infty_c({\mathbf R}^n)$ be such that
${\rm supp~} \phi  \subset B_1(x_0)$ and
$\phi (x) =1$ on $B_{1/2}(x_0)$,
and define
$$
f(x) = \big( (x_1 - x_1^{(0) } )   + \cdots + (x_n -x^{(0)}_n ) \big) g(x),
\quad
x= (x_1, \cdots, x_n) \in {\mathbf R}^n.
$$
Then  $f \in  I(\{ x_0 \})  $.
Since  $I(\{ x_0 \})  = J(\{ x_0 \})$, 
we see that for any $\varepsilon >0$, there exists $f_\varepsilon \in {\mathcal F}L^q_s ({\mathbf R}^n)$
such that $f_\varepsilon (x) =0$ in a neighborhood of $x_0$ and 
$
\| f - f_\varepsilon \|_{{\mathcal F}L^q_s} < \varepsilon
$. 
On the  other hand, 
$f,f_\varepsilon \in L^1({\mathbf R}^n) \cap {\mathcal F}L^1({\mathbf R}^n) $.
Applying the Fourier inversion formula 
and the H\"older inequality one has
\begin{align*}
&\Big|   \frac{\partial f}{\partial x_1} (x_0) - \frac{\partial f_\varepsilon}{\partial x_1} (x_0)    \Big|
=
(2 \pi)^{-n}
\Big|
\int_{{\mathbf R}^n} i \xi_1 (\widehat{f} (\xi)  - \widehat{f_\varepsilon}(\xi))
e^{ix \xi} d \xi
\Big| \\
& \leq 
(2 \pi)^{-n}  \int_{{\mathbf R}^n}   \langle \xi \rangle^{1-s}
\langle \xi \rangle^{s}
| \widehat{f} (\xi)  - \widehat{f_\varepsilon} (\xi) | d \xi \leq 
(2 \pi)^{-n}
\varepsilon
\| \langle \cdot \rangle^{1-s} \|_{L^{q^\prime}} .
\end{align*}
Since 
$ \frac{\partial f}{\partial x_1} (x_0) =1$,
$\frac{\partial f_\varepsilon}{\partial x_1} (x_0) =0$
and $\varepsilon>0$ is arbitrary,
this gives a contradiction.
Hence, $\{ x_0 \}$ is not  a set of spectral synthesis 
for ${\mathcal F}L^q_s({\mathbf R}^n)$.\\

To prove Theorem \ref{spectral synthesis for Mpqs},
we prepare several lemmas.
In the following, 
$({\mathcal F}L^q_s)_c$ 
denotes the space defined in
Lemma \ref{inclusion relation (FLqs)_c}.
\begin{lemma}
\label{ideal theorem lemma 1}
Let $q=1$ and $s \geq 0$, or $1<q<\infty$ and $s> n / q^\prime$.
Suppose that  $I$ is a closed ideal in ${\mathcal F}L^q_s({\mathbf R}^n)$.
Then  $ I =
\overline{ I \cap ( {\mathcal F}L^q_s)_c}^{\|  \cdot \|_{{\mathcal F}L^q_s}}. $
\end{lemma}

\begin{proof}
We need only consider
$I \subset   \overline{ I \cap ( {\mathcal F}L^q_s)_c}^{\|  \cdot \|_{{\mathcal F}L^q_s}}$.  
Let $f \in I$ and $\varepsilon >0$.
By Lemma  \ref{approximate unit for FLqs}
there exists $\phi \in C^\infty_c({\mathbf R}^n)$
satisfying  $\| \phi f    - f \|_{{\mathcal F}L^q_s} < \varepsilon$.
Since $C^\infty_c ({\mathbf R}^n) \subset {\mathcal F}L^q_s ({\mathbf R}^n)$ and 
$I$ is an ideal in ${\mathcal F}L^q_s({\mathbf R}^n)$, we get
$\phi f \in I \cap ( {\mathcal F}L^q_s )_c$,
and thus $I \subset \overline{ I \cap ( {\mathcal F}L^q_s)_c}^{\|  \cdot \|_{{\mathcal F}L^q_s}}$.
\end{proof}
\begin{lemma}
\label{ideal theorem lemma 2}
Let $1 \leq p \leq 2$.
Suppose that $q=1$ and $s \geq 0$, or $1<q \leq p^\prime$ and 
$s> n / q^\prime$.
For  closed ideals $I$ and $I^\prime$ in
${\mathcal F}L^q_s({\mathbf R}^n)$, 
$(i)$ $I \cap M^{p,q}_s({\mathbf R}^n)$ is a closed ideal
in $M^{p,q}_s({\mathbf R}^n)$.
$(ii)$
If $I \cap M^{p,q}_s({\mathbf R}^n)
= I^\prime \cap M^{p,q}_s({\mathbf R}^n)$,
then we have $I=I^\prime$.
\end{lemma}
\begin{proof}
$(i)$  Recall that $M^{p,q}_s ({\mathbf R}^n)$ is a multiplication algebra.
Moreover $M^{p,q}_s({\mathbf R}^n)  \hookrightarrow {\mathcal F}L^q_s({\mathbf R}^n)$
(see Lemma \ref{inclusion Mpqs subset FLqs})
and $I$ is a ideal in ${\mathcal F}L^q_s({\mathbf R}^n)$.
Thus $I \cdot M^{p,q}_s \subset {\mathcal F}L^q_s \cdot I \subset I$.
Hence $I \cap M^{p,q}_s$ is an ideal in $M^{p,q}_s({\mathbf R}^n)$.
To see that  $I \cap M^{p,q}_s({\mathbf R}^n)$ is  closed,
let $f \in  \overline{ I \cap M^{p,q}_s({\mathbf R}^n)}^{\| \cdot \|_{M^{p,q}_s}} $.
Then there exists $\{ f_n \}^\infty_{n=1} \subset I \cap M^{p,q}_s({\mathbf R}^n)$
such that $\| f_n - f \|_{M^{p,q}_s} \to 0$ $(n \to \infty)$.
Thus Lemma \ref{inclusion Mpqs subset FLqs} gives $\| f_n - f \|_{{\mathcal F}L^q_s} \to 0$ $(n \to \infty)$.
Moreover, since  $M^{p,q}_s({\mathbf R}^n)$ is complete
and $I$ is closed in ${\mathcal F}L^q_s({\mathbf R}^n)$,
we have  $f \in I \cap M^{p,q}_s({\mathbf R}^n)$,
which shows $ I \cap M^{p,q}_s({\mathbf R}^n)$ is closed.
$(ii)$ Since $I \cap M^{p,q}_s  \cap ( {\mathcal F}L^q_s )_c
= I^\prime \cap M^{p,q}_s \cap ( {\mathcal F}L^q_s )_c
$
and $({\mathcal F}L^q_s )_c \hookrightarrow M^{p,q}_s ({\mathbf R}^n)$
(see Lemma \ref{inclusion relation (FLqs)_c}),
one has $I \cap ({\mathcal F}L^q_s )_c =  I^\prime \cap ({\mathcal F}L^q_s )_c$.
Thus  Lemma \ref{ideal theorem lemma 1} yields  $I=I^\prime$. 
\end{proof}
\begin{proposition}
\label{ideal theorem}
Let $1 \leq p \leq 2$. Suppose that
$q=1$ and $s \geq 0$, or $1<q \leq p^\prime$ and
$s> n / q^\prime$.
For any closed ideal $I_M$ in $M^{p,q}_s({\mathbf R}^n)$,
the ideal $I_F  := \overline{I_M}^{ \| \cdot \|_{{\mathcal F}L^q_s} } $ in ${\mathcal F} L^q_s ({\mathbf R}^n)$
satisfies $I_M = I_F \cap M^{p,q}_s ({\mathbf R}^n)$.
\end{proposition}
\begin{proof}
We start by observing that 
$I_F^\prime := \overline{ I_M \cap ({\mathcal F}L^q_s)_c  }^{ \| \cdot \|_{{\mathcal F}L^{q}_s} }$
 is a closed ideal in ${\mathcal F}L^q_s ({\mathbf R}^n)$.
In fact,  for $f \in I_F^\prime$ and $g \in {\mathcal F}L^q_s ({\mathbf R}^n)$ there exists $\{ f_n \}^\infty_{n=1}$ in  $I_M \cap ({\mathcal F}L^q_s)_c$
such that $\| f - f_n \|_{{\mathcal F}L^q_s} \to 0$ $ (n \to \infty)$.
Since $f_n \in ( {\mathcal F}L^q_s )_c$,
there exists $\psi_n \in C^\infty_c ({\mathbf R}^n)$
such that $\psi_n(x) =1$ on ${\rm supp~} f_n$.
Then we have
$\psi_n g \in ( {\mathcal F}L^q_s )_c \hookrightarrow M^{p,q}_s({\mathbf R}^n)$.
Therefore, $\psi_n g \cdot f_n \in I_M$, and thus 
$f_ng = f_n \cdot \psi_n g \in I_M \cap ({\mathcal F}L^q_s)_c$.
Furthermore, since 
$
\| fg -f_n g \|_{{\mathcal F}L^q_s} \lesssim
\| f - f_n \|_{{\mathcal F}L^q_s} \| g \|_{{\mathcal F}L^q_s}
\to 0 \ (n \to \infty),
$
and thus $fg \in I_F^\prime$. Hence, $I_F^\prime$ is an ideal in
${\mathcal F}L^q_s ({\mathbf R}^n)$.
We next  prove $I_F = I_F^\prime$.
It suffices to prove $I_F \subset I_F^\prime$.
Given $f \in I_F$ and  $\varepsilon >0$ 
there exists $g \in I_M$ such that 
$\| f -g \|_{{\mathcal F}L^q_s} < \varepsilon$.
By Lemma \ref{approximate unit for Mpqs}
there exists $\phi \in C^\infty_c({\mathbf R}^n)$
such that $\| g - \phi g \|_{M^{p,q}_s} < \varepsilon$.
We note that $\phi g \in I_M \cap ({\mathcal F}L^q_s)_c $
by $\phi g \in I_M \subset M^{p,q}_s ({\mathbf R}^n) 
\hookrightarrow {\mathcal F}L^q_s ({\mathbf R}^n)$.
Since
$
\| f - \phi g \|_{{\mathcal F}L^q_s}
\lesssim
\| f - g \|_{{\mathcal F}L^q_s} + \| \phi g - g \|_{M^{p,q}_s} 
\lesssim \varepsilon, $
we obtain $I_F \subset I_F^\prime$.
Finally, we prove 
$I_M = I_F \cap M^{p,q}_s({\mathbf R}^n)$.
It suffices to show $I_F \cap M^{p,q}_s({\mathbf R}^n) \subset I_M$.
Let $f \in I_F \cap M^{p,q}_s({\mathbf R}^n)$
and $\varepsilon >0$. 
By Lemma \ref{approximate unit for Mpqs}
there exists $\phi \in C^\infty_c({\mathbf R}^n)$
such that $\| f -\phi f \|_{M^{p,q}_s} < \varepsilon$.
Take $\varphi \in C^\infty_c({\mathbf R}^n)$
with $\varphi(x)=1$ on ${\rm supp~} \phi $.
Since
$f \in I_F$ there exists $h \in I_M \cap ( {\mathcal F}L^q_s )_c$
such that  
$\| f - h \|_{{\mathcal F}L^q_s} <
\varepsilon /   (\| \varphi \|_{M^{p,q}_s} \| \phi \|_{{\mathcal F}L^q_s} )$.
Then $\phi h \in I_M$.
Note that the proof of Lemma \ref{inclusion relation (FLqs)_c}
implies ${\mathcal F}L^q_s ({\mathbf R}^n) \hookrightarrow 
M^{\infty,q}_s ({\mathbf R}^n)$.
Thus
\begin{align*}
\| f - \phi h \|_{M^{p,q}_s}
& \leq
\| f - \phi f \|_{M^{p,q}_s} + \| \varphi \phi (f-h) \|_{M^{p,q}_s} \\
&
\lesssim 
\| f - \phi f \|_{M^{p,q}_s} + \| \varphi \|_{M^{p,q}_s} \| \phi (f-h) \|_{M^{\infty,q}_s} \\
&
\lesssim 
\| f - \phi f \|_{M^{p,q}_s} + \| \varphi \|_{M^{p,q}_s} \| \phi (f-h) \|_{{\mathcal F}L^q_s} 
\\
&
\lesssim 
\| f - \phi f \|_{M^{p,q}_s} + \| \varphi \|_{M^{p,q}_s} \| \phi \|_{{\mathcal F}L^q_s} 
\| f-h \|_{{\mathcal F}L^q_s}.
\end{align*}
Therefore $f  \in \overline{ I_M }^{ \| \cdot \|_{M^{p,q}_s}  } = I_M $.
Hence
$I_F \cap M^{p,q}_s({\mathbf R}^n) \subset I_M$.
\end{proof}
\begin{remark}
Let $I_M$ and $I^\prime_M$
be closed ideals in $M^{p,q}_s ({\mathbf R}^n)$, and
$I_F$ be the closure of $I_M$ in ${\mathcal F}L^q_s ({\mathbf R}^n)$.
If the closure of $I^\prime_M$ in ${\mathcal F}L^q_s ({\mathbf R}^n)$
is equal to $I_F$, then
Proposition \ref{ideal theorem} implies that
$I_M = I^\prime_M$. 
\end{remark}

Combining these results, we obtain the
``ideal theory for Segal algebras''.
\begin{theorem}
\label{ideal theorem of segal algebra}
Let $1 \leq p \leq 2$. Suppose that
$q=1$ and $s \geq 0$, or $1<q \leq p^\prime$ and
$s> n / q^\prime$.
Let ${\mathcal I}_F$ be the set of all closed ideals in ${\mathcal F}L^q_s ({\mathbf R}^n)$,
and ${\mathcal I}_M$ be the set of all closed ideals in 
$M^{p,q}_s({\mathbf R}^n)$. 
Then the map $\iota : {\mathcal I}_F \to 
{\mathcal I}_M$, $\iota(I_F) = I_F \cap M^{p,q}_s ({\mathbf R}^n) $
$(I_F \in {\mathcal I}_F)$
is bijective.
More precisely, we have
$\iota^{-1} (I_M) = \overline{I_M}^{ \| \cdot \|_{{\mathcal F}L^q_s}  }$
and
$
\iota (   \overline{I_M}^{ \| \cdot \|_{{\mathcal F}L^q_s}  } ) 
=  \overline{I_M}^{ \| \cdot \|_{{\mathcal F}L^q_s}  }
 \cap M^{p,q}_s ({\mathbf R}^n)
$
for $I_M \in {\mathcal I}_M $. 
\end{theorem}

\subsubsection{The proof of Theorem \ref{spectral synthesis for Mpqs}}
For a closed subset $K$ of  ${\mathbf R}^n$,
we set $I_F(K) :=  \{ f \in {\mathcal F}L^q_s ({\mathbf R}^n) ~|~ f|_K =0 \}$ 
and $I_M(K) :=  \{ f \in M^{p,q}_s ({\mathbf R}^n) ~|~ f|_K =0 \}$.
Moreover,  we define $J_F(K)$ by the closure of 
$
 \{ f \in {\mathcal F}L^q_s ({\mathbf R}^n) ~|~ f (x) =0 {\rm ~in~ a~neighborhood~of~} K  \}
$
 in ${\mathcal F}L^q_s({\mathbf R}^n)$, 
and 
$J_M(K)$ by the 
closure of 
$
 \{ f \in M^{p,q}_s ({\mathbf R}^n) ~|~ f(x) =0 {\rm ~in~ a~neighborhood~of~} K  \}$
 in $M^{p,q}_s({\mathbf R}^n)$.
Then Theorem \ref{ideal theorem of segal algebra} shows
$I_F(K) = J_M(K)$ if and only if 
$I_F(K) = J_F(K) $.
Hence, $K$ is a set of spectral synthesis for $M^{p,q}_s({\mathbf R}^n)$ if and only if $K$ is a set of spectral synthesis for ${\mathcal F}L^q_s ({\mathbf R}^n)$. 

\section{Wiener-L\'evy theorem}
\label{section Wiener Levy theorem}
We only prove Theorem \ref{wiener-levy for FLqs}
because a slight change in the proof of 
Theorem \ref{wiener-levy for FLqs}
shows Theorem \ref{wiener-levy for Mpqs}
(cf. Remark \ref{rem modulation version}).
We first recall the following lemma
and prepare a local version of Theorem \ref{wiener-levy for FLqs}.

\begin{lemma}[{\cite[Theorem 4.13]{Reich Sickel}}]
\label{composition on Mpqs}
Let $1<p<\infty$, $1 \leq q \leq \infty$ and $s>n/q^\prime$.
Let $\mu$ be a complex measure on $\mathbf R$
such that
$$
\int_{\mathbf R}
(1+|\xi|)^{ 1+(s+n/q)  (1+ \frac{1}{ s- n/q^\prime })  } d |\mu| (\xi) < \infty
$$
and such that $\mu({\mathbf R}) =0$.
Let $F$ be the inverse Fourier 
transform of $\mu$. 
Then $F(f) \in M^{p,q}_s({\mathbf R}^n)$ holds for all real-valued 
$f \in M^{p,q}_s({\mathbf R}^n)$.  
\end{lemma}
\begin{lemma}
\label{local wiener-levy for FLqs}
Given $1<q<\infty$, $s> n/q^\prime$,
a real-valued function
$f \in  {\mathcal F}L^q_s({\mathbf R}^n)$
and a compact subset
$K \subset {\mathbf R}^n$.
Suppose that $F \in {\mathcal S}({\mathbf R})$ and
$F(0)=0$. Then there exists $g \in {\mathcal F}L^q_s({\mathbf R}^n)$
such that $g(x) = F(f(x))$ for all $x \in K$.
\end{lemma}
\begin{proof}
Take a real-valued function $\tau \in C^\infty_c ({\mathbf R}^n) $ with 
$\tau  (x) =1$ on $K$.
Then $\tau f \in {\mathcal F}L^q_s ({\mathbf R}^n)$.
Since
\begin{equation}
\label{equality of compact supported function space}
({\mathcal F}L^q_s)_c
= \{ f \in  M^{p,q}_s({\mathbf R}^n) ~|~ {\rm supp~} f ~ {\rm is~compact} \}
\end{equation}
(cf. \cite[Lemma A.1]{Kato Sugimoto Tomita}, \cite[Lemma 1]{Okoudjou}), 
we have
 $\tau f \in M^{p,q}_s ({\mathbf R}^n)$
 and thus $F( \tau f  ) \in M^{p,q}_s ({\mathbf R}^n)$
 by Lemma \ref{composition on Mpqs}.
Note that
$F \in {\mathcal S} ({\mathbf R})$
and ${\rm supp}(\tau f)$ is compact.
Thus ${\rm supp} (F(\tau f))$ is compact.
By  \eqref{equality of compact supported function space} 
we have  $F(\tau f) \in {\mathcal F}L^q_s({\mathbf R}^n)$.
Now set $g= F(\tau f)$. 
Then $g \in {\mathcal F}L^q_s ({\mathbf R}^n)$
and
$g(x) = F( \tau (x)  f(x)  ) = F(f(x))$ $(x \in K)$.
\end{proof}

\subsection{The proof of Theorem \ref{wiener-levy for FLqs}}

Since $F$ is analytic on a neighborhood of $0$ with $F(0)=0$,
there exists $\varepsilon_0>0$ such that
$F(z)$ has the power series representation
$ F(z) = \sum^\infty_{j=1} c_j z^j $
$ (|z| < \varepsilon_0).$
Take $\phi \in C^\infty_c({\mathbf R}^n)$ such that
$ \| f - \phi f \|_{ {\mathcal F}L^q_s }< \varepsilon /c $
for any $\varepsilon$ with $0< \varepsilon < \varepsilon_0$
(see Lemma \ref{approximate unit for FLqs}),
where $c$ is the constant as in \eqref{multiplication constant for FLqs}.
Now we set  $
g_0(x) := \sum^\infty_{j=1} c_j \big( f (x) - \phi (x) f (x)  \big)^j. $
Then $g_0 \in  {\mathcal F}L^q_s({\mathbf R}^n)$ and
$g_0(x) = F( f(x) - \phi(x) f(x))   =F(f(x))$ $(x \not\in {\rm supp~}\phi)$.
On the other hand, let $\tau_0, \tau_1 \in C^\infty_c({\mathbf R}^n)$ be
such that $\tau_0 (x) =1$
 on ${\rm supp~} \phi$ and $\tau_1(x) =1$ on ${\rm supp~} \tau_0$.
Take $\psi  \in C^\infty_c({\mathbf R}^n)$
with $\psi (y) =1$ for all $ y \in  \{\tau_1(x) f(x)  ~|~  x \in {\rm supp~} \tau_1 \}$.
We note $f \in 
L^\infty({\mathbf R}^n)$.
Moreover one has
$G := \psi F \in C^\infty_c({\mathbf R})$, $G (0)=0$
 and
$$
F(f(x))
= \psi ( \tau_1 (x) f(x) )F( \tau_1(x) f(x)  ) 
= G( \tau_1 (x) f(x) )
\quad
(x \in {\rm supp~} \tau_0).
$$
By Lemma \ref{local wiener-levy for FLqs} with 
$K= {\rm supp~} \tau_0$ and
$\tau_1 f \in {\mathcal F}L^q_s({\mathbf R}^n)$,
there exists exists $g_1 \in {\mathcal F}L^q_s({\mathbf R}^n)$
such that
$g_1(x) = F(f(x))$ on  ${\rm supp~} \tau_0$.
Set
$
g(x) :=  (1- \tau_0 (x))g_0 (x) + \tau_0 (x)  g_1 (x).
$
Then $g \in {\mathcal F}L^q_s({\mathbf R}^n)$.
If $x \in {\rm supp~} \phi$, then
$\tau_0(x)=1$ and $g_1(x) =F(f(x))$.
Thus $g(x) =F(f(x))$.
Moreover,  if $x  \in {\rm supp~} \tau_0 \setminus   {\rm supp~}\phi$,
then  $g_0(x)  = F( f(x) - \phi (x) f(x)) = F(f(x))$ and $g_1(x) = F(f(x))$.
Thus $g(x) = F(f(x))$. If $x \not\in {\rm supp~} \tau_0$, then
$\tau_0 (x)=0$, $g_0 (x) = F(f(x))$
and
$g(x) = F(f(x))$.
\begin{remark}
\label{rem modulation version}
Applying  Lemma \ref{approximate unit for Mpqs} and
Lemma \ref{composition on Mpqs} to $\tau_1 f$
for a real-valued function  $f \in M^{p,q}_s({\mathbf R}^n)$,
we can prove Theorem \ref{wiener-levy for Mpqs} similarly.
\end{remark}

\section*{Acknowledgment}
This work was supported by JSPS KAKENHI Grant Numbers 22K03328, 22K03331.

\end{document}